\documentclass[12pt,a4paper]{amsart}
\usepackage[all]{xy}
\newtheorem{theorem}{Theorem}%[section]
\newtheorem{lemma}[theorem]{Lemma}
\newtheorem{prop}[theorem]{Proposition}
\newtheorem{cor}[theorem]{Corollary}
\newtheorem{definition}[theorem]{Definition}
\theoremstyle{remark}
\newtheorem{example}[theorem]{Example}
\newtheorem{remark}[theorem]{Remark}

\newtheorem{problem*}{Problem}
\newtheorem{remark*}{Remark}
\newtheorem{convention*}{Convention}
\newtheorem{notation*}{Notation}
\newtheorem{examples*}{Examples}
\newtheorem{example*}{Example}
\newtheorem{warning*}{Warning}

\def\R{{\mathbb R}}

\def\Z{{\mathbb Z}}

\begin{document}
\title{Classification of Generalized Multiresolution Analyses}
\author[L.
 W.
 Baggett]{Lawrence~W.
~Baggett}
\address{Lawrence Baggett, Department of Mathematics, University of Colorado, Boulder, Colorado 80309, USA}
\email{baggett@euclid.colorado.edu}
\author[V.
 Furst]{Veronika~Furst}
\address{Veronika Furst, Department of Mathematics, Fort Lewis College, Durango, Colorado 81301, USA}
\email{furst\_v@fortlewis.edu}
\author[K.
 D.
 Merrill]{Kathy~D.
~Merrill}
\address{Kathy Merrill, Department of Mathematics, Colorado College, Colorado Springs, Colorado, 80903, USA}
\email{kmerrill@coloradocollege.edu}
\author[J.
 A.
 Packer]{Judith~A.
~Packer}
\address{Judith Packer, Department of Mathematics, University of Colorado, Boulder, Colorado 80309, USA}\email{packer@euclid.colorado.edu}
\begin{abstract}
We discuss how generalized multiresolution analyses (GMRAs), both classical and those defined
on abstract Hilbert spaces, can be classified by their multiplicity
functions $m$ and matrix-valued filter functions $H$.  Given a
natural number valued function $m$ and a system of functions encoded in a
matrix $H$ satisfying certain conditions, a construction procedure is
described that produces an abstract GMRA with multiplicity function $m
$ and filter system $H$. An equivalence relation on GMRAs is defined
and described in terms of their associated pairs $(m,H)$.  This
classification system is applied to classical examples in $L^2
(\mathbb R^d)$ as well as to previously studied abstract examples.
\end{abstract}

\thanks{This research was supported by the National Science Foundation through grant DMS-0701913. }

%beginning
\maketitle
\section{Introduction}
\par A generalized multiresolution analysis (GMRA) is a Hilbert space structure traditionally associated with classical wavelets, that is, functions whose dilates of translates provide an orthonormal basis for $L^2(\mathbb R^d)$.
  Given a wavelet, the nested sequence of subspaces $V_j$ that result from taking only dilation powers less than $j$ are dense and have trivial intersection, with $V_{j+1}$ the dilate of $V_j$, and with $V_0$ invariant under translation.
  Such a structure is called a GMRA (\cite{BMM}), and was developed to understand such wavelets as the famous example given by Journe, whose $V_0$ space does not have an orthornormal basis given by translates of a single function called a scaling function.
  When $V_0$ has this stronger property, the nested sequence $\{V_j\}$ is called an MRA (\cite{Ma},\cite{Me}).
  Both MRAs and GMRAs have been extensively exploited to produce and understand wavelets, which in turn have proven useful for applications such as image and signal processing.
   
\par While wavelets and multiresolution structures were first studied in the Hilbert space $L^2(\mathbb R^d)$, analogous definitions make sense in other Hilbert spaces that have appropriate dilation and translation operators.
  Dutkay and Jorgensen (\cite{DJ}) pioneered the study of wavelets in function spaces on fractals, with later work by D'Andrea et al.(\cite{DMP}).
 Larsen, Raeburn and coworkers then showed that these and other interesting examples can be constructed via direct limits (\cite{IN}, \cite{IJKLN}, \cite{AIJLN}).  Dutkay et al.(\cite{BDP},\cite{DJ2}) constructed MRAs and super-wavelets in Hilbert spaces formed by direct sums of $L^2(\mathbb R^d)$ to orthonormalize examples such as the Cohen wavelet.    
  Tensor products of known examples lead to more exotic specimens (see Section \ref{examples}).    Our purpose in this paper is to construct a set of classifying parameters for GMRAs in order to unify and allow comparison of all these disparate examples.  We also provide an explicit construction of a canonical GMRA equivalent to each of them.
\par Accordingly, we will consider GMRA structures in an abstract  Hilbert space $\mathcal H$, equipped with ``translations" given by a unitary representation $\pi$ of a countable abelian group $\Gamma$ acting in $\mathcal H$, and a ``dilation" given by a unitary operator $\delta$.
   
We assume that these operators are related by 
\begin{equation}
\label{piconj}
\delta^{-1}\pi_\gamma \delta = \pi_{\alpha(\gamma)}
\end{equation}
for all $\gamma\in \Gamma,$ where $\alpha$ is an isomorphism of $\Gamma$ into itself such 
that the index of $\alpha(\Gamma)$ in $\Gamma$ equals $N>1,$ and such that $\cap \alpha^n(\Gamma) = \{0\}.
$  These definitions generalize the classical case of ordinary translation by the integer lattice in $L^2(\mathbb R^d)$, given by $\pi_n f(x)=f(x-n)$, and dilation by an expansive integer matrix $A$, given by $\delta f(x)=\sqrt{|\det A|}f(Ax)$.
\par The structure of a GMRA, and thus the parameters that uniquely identify it, are revealed via Stone's Theorem on unitary representations of abelian groups.
  Using this theorem, we know that the representation $\pi$ restricted to $V_0$ is completely determined by a measure $\mu$ on the dual group $\widehat\Gamma$  and a Borel multiplicity function $m:\widehat\Gamma\to \{0,1,2,\cdots,\infty\}$, which essentially describes how many times each character occurs in the decomposition of $\pi|_{V_0}$.
  There is a unitary equivalence $J$ between the action of $\pi$ on $V_0$ and multiplication by characters on 
$\oplus L^2(\sigma_i),$ where $\sigma_i=\{\omega:m(\omega)\geq i\}$.  Because of this, we think of $J$ as a partial alternative Fourier transform.  For simplicity, in this paper we will restrict our attention to the commonly studied case where $\mu$ is Haar measure, and $m$ is finite a.e..

\par The multiplicity function $m$ is one of the parameters that determine a GMRA.
  As we will see in Section \ref{equivalence}, the other parameter 
is a ``filter''that shows how the operator $J$ interacts with dilation.
  Classical filters were periodic functions $h$ and $g$ in $L^2(\mathbb R^d)$ that described inverse dilates of Fourier transforms of bases of $V_1$ in terms of those of $V_0$.
  Starting with an MRA in $L^2(\mathbb R^d)$, such functions could be shown to satisfy certain orthogonality relations.
  Mallat, Meyer and Daubechies (\cite{Ma},\cite{Me},\cite{Da}) turned this process around by using functions $h$ and $g$ satisfying orthogonality together with additional low-pass and non-vanishing conditions to construct MRAs and wavelets.
  Lawton (\cite{Law}) and Bratelli/Jorgensen (\cite{BJ2}) were able to relax the non-vanishing condition by allowing Parseval frames in place of orthonormal bases, and Baggett, Courter, Jorgensen, Merrill, Packer (\cite{BCM},\cite{BJMP}) generalized this work to the GMRA setting by replacing $h$ and $g$ by matrix-valued functions $H$ and $G$.  
  In \cite{BJ}, Bratelli and Jorgensen related filters $h$ and $g$ to Ruelle operators $S_h$ and $S_g$, which satisfy relations similar to those of Cuntz operators, and can be used to represent inverse dilations.  This work was extended to generalized filters in
\cite{BJMP} and later \cite{IJKLN}.   
\par In the next section, we recall the relationship between abstract GMRAs, multiplicity functions and generalized filters.
   In particular, we describe conditions on a multiplicity function $m$ and a filter $H$ that guarantee that they will produce a GMRA.
  It turns out that these conditions are considerably more relaxed in an abstract Hilbert space than in $L^2(\mathbb R^d)$.
 In Section \ref{construction} we describe a construction procedure that produces an abstract
GMRA from any $m$ and $H$ meeting the required conditions.  This construction gives an explicit realization of the abstract direct limit GMRAs built in \cite{IJKLN}.  While the procedure relies on first choosing a filter $G$ complementary to $H$, we show in Section \ref{equivalence} that the equivalence between GMRAs does not depend on the choice of $G$. Thus, the classifying set described there depends only on the pair
$m$ and $H$.  In this section, we also give a necessary and sufficient
condition on the equivalence class of a filter so that $S_H$ is a pure
isometry and thus associated with a GMRA, in the case of a finite multiplicity function.  We conclude in
Section \ref{examples} with a variety of examples that illustrate our
main theorems, including an example of a GMRA where the translation
group is not isomorphic to $\mathbb Z^d$.
%%%%%%%%% Section 2
\section{GMRAs, multiplicity functions and filters}
\label{gmras}
Let $\mathcal H$ be an abstract, separable Hilbert space, equipped with operators $\pi_{\gamma}$ and $\delta$ satisfying Equation (\ref{piconj}).
\begin{definition}
A collection $\{V_j\}_{-\infty}^\infty$ of closed subspaces of $\mathcal H$
is called a {\bf generalized multiresolution analysis} (GMRA) relative to $\pi$ and $\delta$ if
\begin{enumerate}
\item\hskip2em $V_j\subseteq V_{j+1}$ for all $j.
$
\item\hskip2em $V_{j+1}=\delta(V_j)$ for all $j.
$
\item\hskip2em $\cap V_j=\{0\},$ and $\cup V_j$ is dense in $\mathcal H.
$
\item\hskip2em $V_0$ is invariant under the representation $\pi.
$
\end{enumerate}
The subspace $V_0$ is called the {\bf core subspace} of the GMRA $\{V_j\}.
$
\end{definition}
Let $\{V_j\}$ be a GMRA in a Hilbert space $\mathcal H.
$
For each $j,$ write $W_j$ for the orthogonal complement to
$V_j$ in $V_{j+1}.
$
It follows that $\mathcal H = \bigoplus_{j=-\infty}^\infty W_j.
$
Also, for each $j\geq 0,$ $W_j$ is an invariant subspace for the representation $\pi.
$
We apply Stone's Theorem on unitary representations of abelian groups
to the subrepresentations of $\pi$ acting in $V_0$ and $W_0.
$
Accordingly, there exists a
finite, Borel measure $\mu$ (unique up to equivalence of measures) on $\widehat\Gamma,$
Borel subsets $\sigma_1\supseteq \sigma_2 \supseteq \ldots$ of $\widehat\Gamma$ (unique up to sets of $\mu$ measure 0),
and a (not necessarily unique) unitary operator 
$J:V_0 \to \bigoplus_i L^2(\sigma_i,\mu)$ satisfying
\[
[J(\pi_{\gamma}(f))](\omega) = \omega(\gamma) [J(f)](\omega)
\]
for all $\gamma\in\Gamma,$ all $f\in V_0,$ and $\mu$ almost all $\omega\in\widehat\Gamma.
$
We write $m$ for the function on $\widehat\Gamma$ given by
$m(\omega) = \sum_i \chi_{\sigma_i}(\omega),$ and call it
the {\bf multiplicity function} associated to the representation $\pi|_{V_0}.$
\par
Analogously,
there exists a finite, Borel measure $\widetilde\mu,$ Borel subsets $\widetilde\sigma_k,$ and an operator
$\widetilde J:W_0\to \bigoplus_k L^2(\widetilde\sigma_k,\widetilde\mu)$ satisfying
\[
[\widetilde J(\widetilde\pi_\gamma(f))](\omega) = \omega(\gamma) [\widetilde J(f)](\omega)
\]
for all $\gamma\in\Gamma,$ $f\in W_0,$ and $\widetilde\mu$ almost all $\omega.
$
We write $\widetilde m$ for the function on $\widehat\Gamma$ given by
$\widetilde m(\omega) = \sum_k \chi_{\widetilde\sigma_k}(\omega),$ and call it
the multiplicity function associated to the representation $\pi|_{W_0}.$
\par
In this paper, we will assume that the measures $\mu$ and $\widetilde\mu$ are absolutely continuous with respect to Haar measure,
and thus take $\mu$ and $\widetilde\mu$ to be the restrictions of Haar measure to the subsets $\sigma_1$
and  $\widetilde\sigma_1,$ respectively.
We also assume that the multiplicity function $m$ associated to the representation
$\pi|_{V_0}$ is finite almost everywhere.
\par
Let $\alpha^*$ be the dual endomorphism of $\widehat\Gamma$ onto itself defined by
$[\alpha^*(\omega)](\gamma) = \omega(\alpha(\gamma)),$ and note that
the kernel of $\alpha^*$ contains exactly $N$ elements
and that $\alpha^*$ is ergodic with respect to the Haar measure $\mu$ on $\widehat\Gamma.
$
Using $\alpha^*$ to relate the representations $\pi|_{V_1}$ and  $\pi|_{V_0}$, it is shown in \cite{BMM} and more generally in \cite{IJKLN} that multiplicity functions for a GMRA must satisfy the following consistency equation:
\begin{equation}
\label{consistency}
m(\omega) +  \widetilde m(\omega) = \sum_{\alpha^*(\zeta)=\omega} m(\zeta) .
 \end{equation}
It follows, since the function $m$ is finite a.e., that 
the sets $\sigma_i$ and $\widetilde{\sigma}_k$ are completely determined by the multiplicity function $m.$
It also follows that a multiplicity function $m$ associated with a GMRA must satisfy the consistency inequality:
\begin{equation}
\label{conineq}
m(\omega) \leq \sum_{\alpha^*(\zeta)=\omega} m(\zeta).
 \end{equation}

 \par We will see in the next section that the consistency inequality is a sufficient as well as necessary condition for a function  $m:\widehat\Gamma \to \{0,1,2,\ldots\}$ 
to be a multiplicity function associated to an abstract GMRA.
Accordingly, we make the following definition.
\begin{definition}
\label{mult}
A {\bf multiplicity function} is a Borel function $m:\widehat\Gamma\to \{0,1,2,\ldots\}$ that
satisfies the consistency inequality (\ref{conineq}). 
\end{definition}
  In contrast, Bownik, Rzeszotnik and Speegle (\cite{BRS}) and Baggett and Merrill (\cite{BM}) showed that an additional technical condition related to dilates of the translates of the support of $m$ is required for $m$ to be a multiplicity function for a GMRA in $L^2(\mathbb R^d)$.
 We will need the following observation about multiplicity functions.  
  
\begin{prop}\label{strictinequality}
Suppose $m:\widehat\Gamma \to \{0,1,2,\ldots\}$ satisfies the consistency inequality.
If $m$ is not identically 0, then there exists a set
$F$ of positive measure in $\widehat\Gamma$ such that
\[
m(\omega) < \sum_{\alpha^*(\zeta)=\omega} m(\zeta)
\]
for all $\omega\in F.
$  That is, the consistency inequality
is a strict inequality on a set of positive measure.
\end{prop}
\begin{proof} Suppose 
\[
m(\omega) = \sum_{\alpha^*(\zeta)=\omega} m(\zeta)
\]
for almost all $\omega\in\widehat\Gamma.
$
It follows directly by induction that
\[
m(\omega) = \sum_{{\alpha^*}^n(\zeta)=\omega} m(\zeta)
\]
for almost all $\omega.
$
Let $k$ be a positive integer for which there exists a set $E\subseteq\widehat\Gamma$ of positive measure
such that $m(\omega)\leq k$ for all $\omega\in E.
$
Choose $n$ such that $N^n>k.
$  Then, for almost every $\omega\in {\alpha^*}^{-n}(E),$ we have
\[
\sum_{\zeta\in\ker({\alpha^*}^n)} m(\omega\zeta) = m({\alpha^*}^n(\omega)) \leq k,
\]
implying that there exists some $\zeta\in\ker({\alpha^*}^n),$
and a subset $E'\subseteq {\alpha^*}^{-n}(E)$ of positive measure, such that
$m(\omega\zeta)=0$ for all $\omega\in E'.
$
Hence, $m(\omega)=0$ on a set $F$ of positive measure.
But, from the equation
\[
m(\alpha^*(\omega)) = \sum_{\alpha^*(\zeta)=1} m(\omega\zeta),
\]
it follows that the sequence $\{m({\alpha^*}^n(\omega))\}$ is nondecreasing.
Because $\alpha^*$ is ergodic,
we must have that the sequence $\{{\alpha^*}^n(\omega)\}$ intersects the set $F$
infinitely often for almost all $\omega.
$
Hence $m(\omega)=0\mbox{ a.e.}.$
\end{proof}
   
The other ingredients we will need for our GMRA construction are filters, which are defined in terms of a multiplicity function $m$ as follows: 
\begin{definition}
Let $m$ be a multiplicity function, and write
$\sigma_i = \{\omega: m(\omega) \geq i\}.$
Set
\[
\widetilde m(\omega) = \sum_{\alpha^*(\zeta)=\omega} m(\zeta) - m(\omega),
\]
and set $\widetilde\sigma_k = \{\omega : \widetilde m(\omega) \geq k\}.$
Let  $H=[h_{i,j}]$ and $G=[g_{k,j}]$ be (possibly infinite) matrices  of Borel, complex-valued functions
on $\widehat\Gamma$ such that for every $j$, $h_{i,j}$ and $g_{k,j}$ are supported in $\sigma_j.
$ Suppose further that $H$ and $G$ satisfy the following ``filter equations:''
\begin{equation}\label{filter1}
\sum_{\alpha^*(\zeta)=\omega} \sum_j h_{i,j}(\zeta)\overline{h_{i',j}(\zeta)}
= N\delta_{i,i'}\chi_{\sigma_i}(\omega),
\end{equation}
\begin{equation}\label{filter2}
\sum_{\alpha^*(\zeta)=\omega} \sum_j g_{k,j}(\zeta)\overline{g_{k',j}(\zeta)}
= N\delta_{k,k'}\chi_{\widetilde{\sigma}_k}(\omega), \mbox{ and}
\end{equation}
\begin{equation}
\label{filter3}
\sum_{\alpha^*(\zeta)=\omega} \sum_j g_{k,j}(\zeta) \overline{h_{i,j}(\zeta)}
=0.
\end{equation}
Then $H$ is called a {\bf filter} relative to $m$ and $\alpha^*$,
and $G$ is called a {\bf complementary filter} to $H.$
\end{definition}
\par
We note that it will sometimes be useful to consider filters and complementary filters to be matrix valued functions on $\widehat{\Gamma}$ rather than a matrix of complex valued functions.
  It is then a consequence of the definition above that the nonzero portion of the matrix $H(\omega)$ is contained in the upper left block of dimensions $m(\alpha^*(\omega))\times m(\omega)$, while the nonzero portion of $G(\omega)$ is contained in the upper left block of dimensions $\widetilde m(\alpha^*(\omega))\times m(\omega)$.
\par
Given a filter $H$ relative to $m$ and $\alpha^*,$
we may define a ``Ruelle'' operator $S_H$ on
$\bigoplus_i L^2(\sigma_i)$ by
\[
[S_H(f)](\omega) = H^t(\omega) f(\alpha^*(\omega)).
\]
Similarly, a complementary filter $G$ defines a Ruelle operator $S_G$ from 
$\bigoplus_i L^2(\widetilde{\sigma}_i)$ to $\bigoplus_i L^2(\sigma_i)$ by
\[
[S_G(f)](\omega) = G^t(\omega) f(\alpha^*(\omega)).
\]
The filter equations satisfied by $H$ and $G$ translate to the following Cuntz-like conditions for the Ruelle operators (see \cite{BJMP},\cite{IJKLN}):
\begin{lemma}
\label{Ruelle}
If $H$ is a filter relative to $m$ and $\alpha^*$, and $G$ is a complementary filter relative to $\widetilde m$ and $H,$ then the Ruelle operators they define satisfy
\begin{enumerate}
\item$ S_H^*S_H = I$, $S_G^*S_G = \widetilde I$
\item $S_H^*S_G =0,$ and
\item $S_HS_H^* + S_GS_G^*=I,$
\end{enumerate}
where $I$ is the identity operator on $\bigoplus_i L^2(\sigma_i,\mu)$
and $\widetilde I$ is the identity operator on $\bigoplus_k L^2(\widetilde\sigma_k,\widetilde\mu).
$
\end{lemma}
\par
Filters, like multiplicity functions, arise naturally out of GMRAs.
  Let $\{V_j\}$ be a GMRA (with finite multiplicity function and associated measure absolutely continuous with respect to Haar),
and let $\mu,\{\sigma_i\},$ $J$, $\widetilde\mu,\{\widetilde\sigma_k\},$ and $\widetilde J$ be as in the Stone's Theorem discussion 
above.
Write $C_i$ for the element of the direct sum space
$\bigoplus_j L^2(\sigma_j,\mu)$ whose $i$th coordinate is $\chi_{\sigma_i}$
and whose other coordinates are 0, and $\widetilde C_k$ for the element in
$\bigoplus_l L^2(\widetilde\sigma_l,\widetilde\mu)$ whose $k$th coordinate is
$\chi_{\widetilde\sigma_k}$ and whose other coordinates are 0.
Let  $\bigoplus_j h_{i,j}$ be the element $J(\delta^{-1}(J^{-1}(C_i)))$ and $\bigoplus_j g_{k,j}$ be the element
$J(\delta^{-1}({\widetilde J}^{-1}(\widetilde C_k))),$ both in $\bigoplus_j L^2(\sigma_j,\mu).
$
It was shown in \cite{BFMP1}  that the matrix $H=[h_{i,j}]$ is then a filter relative to $m$ and $\alpha^*,$ 
and the matrix  $G = [g_{k,j}]$ is a complementary filter to $H.$  We call these {\bf filters constructed from a GMRA}, and note that they are not unique, but rather depend on the choice of the maps $J$ and $\widetilde J$.  
The operators $J\circ \delta^{-1}  \circ J^{-1}$  and  $J\circ \delta^{-1} \circ {\widetilde J}^{-1}$ 
are the corresponding Ruelle operators $S_H$ and $S_G$ respectively.
   It follows directly from their definitions that $S_H$ and $S_G$ are isometries, and the GMRA requirement that $\cap V_j=\{0\}$ implies that $S_H = J \circ \delta^{-1} \circ J^{-1}$ is a pure isometry.
   
Just as with multiplicity functions, this necessary condition on a filter to be associated with a GMRA turns out to be sufficient as well.
  In Theorem 5 of \cite{IJKLN}, it is shown that if $S_H$ is a pure isometry on a Hilbert space $\bigoplus L^2(\sigma_i)$, then it is possible to construct a generalized multiresolution analysis via a direct limit process.
    Our construction in the next section will give a concrete realization under the same hypotheses.
Again, as with multiplicity functions, we see that this necessary and sufficient condition on the filter $H$ is much weaker than what is required for a filter to be associated with a GMRA in $L^2(\mathbb R^d)$.
  For example, in that context, the ``refinement equation," 
\begin{equation}
\label{refine}
\widehat{\phi}(\omega)=\frac1{\sqrt{|\det A|}}H({A^t}^{-1}\omega)\widehat{\phi}({A^t}^{-1}\omega),
\end{equation}
 suggests some sort of convergence of the infinite product $\Pi_{j=1}^{\infty}\frac1{\sqrt{|\det A|}}H((A^t)^{-j}\omega),$ which in turn requires that the filter $H$ satisfy some low-pass condition of being close to $\sqrt{|\det A|}$ times a partial identity near the origin (\cite{BJMP},\cite{IJKLN}).
  Theorems from \cite{BFMP1} and \cite{AIJLN} indicate that in the abstract setting, a much weaker condition is sufficient to guarantee that $S_H$ is a pure isometry.   In the case where the matrix $H$ is $1\times 1$, the simple condition that $|H(\omega)|\neq 1$ on a set of positive measure is sufficient to show that $S_H$ is a pure isometry (\cite{BJ}, \cite{AIJLN}).  
In particular, filters traditionally labeled ``high-pass" can be used as $H$.  Theorem \ref{pure} in Section \ref{equivalence} of this paper gives a new, more general result of this type.  
%%%%%%%Section 3
%construction section
\section{Explicit Construction of GMRAs on abstract Hilbert spaces}
\label{construction}
Let $m$ be a multiplicity function on $\widehat\Gamma,$ as in Definition \ref{mult}
and let $H$ be a filter relative to $m$ and $\alpha^*.$  Using Proposition \ref{strictinequality}, 
define
\begin{equation}
\label{mtilde}
\widetilde m(\omega) = \sum_{\alpha^*(\zeta)=\omega} m(\zeta) - m(\omega),
\end{equation}
and define the sets $\{\sigma_i\}$ and $\{\widetilde\sigma_k\}$ as in the preceding section.  As is shown in \cite{BCM}, given a filter $H$ relative to $m$ and $\alpha^*$, there always exists a complementary filter $G$.  For the purposes of this construction, let $G$ be any filter complementary to $H$.  
\par
Because a $J$ map guaranteed by Stone's theorem takes the core subspace $V_0$ of any GMRA to $\bigoplus L^2(\sigma_i)$, this direct sum of $L^2$ spaces is a natural candidate for the core subspace of an abstract GMRA built out of a multiplicity function and a filter.
The group  $\Gamma$ acts on this space in a natural way via multiplication by characters.
  Similarly, the space $\bigoplus L^2(\widetilde{\sigma}_k)$ is an obvious candidate for the abstract $W_0=V_1\ominus V_0$,  and the relationships  $J\circ \delta^{-1}  \circ J^{-1}=S_H$  and  $J\circ \delta^{-1} \circ {\widetilde J}^{-1}=S_G$ suggest that 
the Ruelle operators $S_H$ and $S_G$ provide natural abstract inverse dilations on these spaces.
 Thus the main task remaining in building a GMRA given $m$ and $H$ is to describe positive dilates of $W_0$; such subspaces could then be used to fill out the rest of the Hilbert space.
\par
If the group $\Gamma=\mathbb Z^d,$ then embedding $\widehat{\Gamma}=\mathbb T^d$ as $[-\frac 12,\frac 12]^d$ in $\mathbb R^d$ provides us with a simple candidate for the dilate of our constructed $W_0=\bigoplus L^2(\widetilde{\sigma}_k)$.
  Since in this case, $\alpha$ is an isomorphism of $\mathbb Z^d$, we must have $\alpha(n)=An$, for a matrix $A$.
  We define positive dilations $\mathcal D^j:\bigoplus L^2(\widetilde{\sigma}_k)\to\bigoplus L^2(A^t\widetilde{\sigma_k})$ by $$\mathcal D^j\left(\bigoplus_k f_k(\omega)\right)=\bigoplus_k \frac1{\sqrt{|\det A|}^j} f_k((A^t)^{-j}\omega).
$$
\par For more general $\Gamma$, we will use an abstract construction to define the positive dilation $\mathcal D$ in terms of a cross section for the map $\alpha^*$.
  Just as in the case of $\Gamma=\mathbb Z^d$, our dilated space will be a direct sum of $L^2$ spaces such that the map $f\to \sqrt{N}f\circ \alpha^*$ determines an isometry of the dilated
space onto the original one.
\par
Let $c$ be a Borel cross-section for the map $\alpha^*;$ i.e., $c$ is a Borel map from $\widehat\Gamma$ into $\widehat\Gamma$ for which
$\alpha^*(c(\omega)) = \omega$ for all $\omega\in\widehat\Gamma.
$
Define $\tau:\widehat\Gamma\to \ker(\alpha^*)$ by
\[
\tau(\omega) = c(\alpha^*(\omega))\omega^{-1}.
\]
\par
Now, let $\nu$ be a finite Borel measure on $\widehat\Gamma.
$
Let $E$ be a Borel subset of $\widehat\Gamma,$
let $\zeta$ be an element of the kernel of $\alpha^*,$
and set
\[
E_{\zeta} = \{\omega\in E:\tau(\omega)=\zeta\}.
\]
\begin{prop}
The set $E$ is the disjoint union $\cup_\zeta E_\zeta,$ and $\alpha^*$ is 1-1
on each $E_\zeta$ into $\widehat\Gamma.
$
\end{prop}
\begin{proof}
The first statement is clear.
\par
For $\omega\in E_\zeta$ we have
\[
c(\alpha^*(\omega)) = \tau(\omega)\omega = \zeta\omega,
\]
which shows that $\alpha^*$ must be 1-1 on $E_\zeta.
$
\end{proof}
Now let $E_1,E_2,\ldots$ be a (countable) collection of Borel
subsets of $\widehat\Gamma.
$
For each $i$  let $\nu_i$ be the restriction to $E_i$ of the measure $\nu.$  Write $E_{i,\zeta}$ for $[E_i]_{\zeta}$,
and let $\nu_{i,\zeta}$ be the restriction of $\nu_i$ to the subset $E_{i,\zeta}$ of $E_i.$
Write $\mathcal K = \bigoplus_i L^2(E_i,\nu_i).
$
For each $\zeta\in\ker(\alpha^*),$ define
$E'_{i,\zeta} = \alpha^*(E_{i,\zeta}),$ and
set $\nu'_{i,\zeta}$ equal to the measure $N\alpha^*_*(\nu_{i,\zeta})$ 
that is defined on $E'_{i,\zeta}$ by
\[
\nu'_{i,\zeta}(F) =N \nu_{i,\zeta}({\alpha^*}^{-1}(F)).
\]
\par
Define
\[
\mathcal K'= \bigoplus_{i,\zeta} L^2(E'_{i,\zeta},\nu'_{i,\zeta}).
\]

\begin{prop}\label{dilationoperator}
For each $f\in \mathcal K,$ set $\mathcal D(f)$
equal to the element of $\mathcal K'$ given by
\[
[\mathcal D(f)]_{i,\zeta}(\omega) = \frac1{\sqrt{N}}f_{i}(\zeta^{-1}c(\omega)).
\]
Then the operator $\mathcal D$ is an isometry of
$\mathcal K$ onto $\mathcal K'.$
\end{prop}
\begin{proof}
\begin{align*}
\|{\mathcal D}(f)\|^2
& = \sum_i\sum_\zeta \int_{E'_{i,\zeta} }|[{\mathcal D}(f)]_{i,\zeta}(\omega)|^2\, d\nu'_{i,\zeta}(\omega) \cr
& =\frac 1N \sum_i\sum_\zeta \int_{E'_{i,\zeta}} |f_i(\zeta^{-1}c(\omega))|^2 \, d\nu'_{i,\zeta}(\omega) \cr
& = \sum_i\sum_\zeta \int_{E_{i,\zeta}} |f_i(\zeta^{-1}c(\alpha^*(\eta))|^2 \, d\nu_{i,\zeta}(\eta) \cr
& = \sum_i\sum_\zeta \int_{E_{i,\zeta}} |f_i(\zeta^{-1}\tau(\eta)\eta)|^2 \, d\nu_{i,\zeta}(\eta) \cr
& = \sum_i\sum_\zeta \int_{i,\zeta} |f_i(\eta)|^2\, d\nu_{i,\zeta}(\eta) \cr
& = \sum_i \int_{E_i}|f_i(\eta)|^2\, d\nu_i(\eta) \cr
& = \|f\|^2,
\end{align*}
where the second to last step is justified because,
for $\eta\in E_{i,\zeta},$ we have $\tau(\eta) = \zeta.$
Thus, $\mathcal D$ is an isometry.
\par
To see that $\mathcal D$ is onto
$\mathcal K',$ it suffices to note that the inverse of $\mathcal D$ is given by
\[
[\mathcal D^{-1}(f)]_i(\omega) =\sqrt N f_{i,\tau(\omega)}(\alpha^*(\omega)).
\]
\end{proof}
We will refer to the space $\mathcal K'={\mathcal D}(\mathcal K)$ as a {\em dilation by} $\alpha^*$  of $\mathcal K$.
  Note that this general definition of $\mathcal D$ is consistent with the definition given at the beginning of this section for the special case of $\Gamma=\mathbb Z^d$.
\par
We are now ready to construct explicitly a GMRA 
from the parameters $m$, $H,$ and $G$.  
\begin{theorem}
\label{gmraconstr}
Suppose $m:\widehat{\Gamma}\rightarrow\{0,1,2,\cdots\}$ is a Borel function that satisfies the consistency inequality, and that $H=[h_{i,j}]$ is a filter relative to $m$  and $\alpha^*$
such that the Ruelle operator $S_H$ is a pure isometry on $\bigoplus_i L^2(\sigma_i).
$  Let $\widetilde{m}$ be defined from $m$ by the consistency equation (as in Equation (\ref{mtilde})), and let $G$ be a complementary filter 
 to $H$.
Define $\mathcal V_0  = \bigoplus_i L^2(\sigma_i)$ and
$\mathcal W_0=\bigoplus_k L^2(\widetilde\sigma_k).
$
For $n\geq 1,$ inductively set $\mathcal W_n={\mathcal D}(\mathcal W_{n-1}),$  and  set $\mathcal H= \mathcal V_0 \oplus\bigoplus_{n=0}^\infty \mathcal W_n.
$
Define a representation $\mathcal\pi$ of $\Gamma,$ acting in $\mathcal H,$ by
\[
[\mathcal\pi_\gamma(f)](\omega) = \omega(\gamma) f(\omega).
\] 
Finally, define an operator $T$ on $\mathcal H$ by
\begin{equation}
[T(f)]_a = \begin{cases}
S_H(f_{\mathcal{V}_0})+S_G(f_{\mathcal{W}_0}) & a=\mathcal{V}_0\\
\mathcal{D}^{-1}(f_{\mathcal W_{n+1}}) & a= \mathcal{W}_{n},\quad n\geq 0
\end{cases},
\end{equation}
where we represent an element $f$ of $\mathcal H$ by $\{f_{\mathcal V_0},f_{\mathcal W_0},f_{\mathcal W_1},\cdots\}$.
  Then
\begin{enumerate}
\item $T$ is a unitary operator on $\mathcal H.$
\item $T\mathcal\pi_\gamma T^{-1} = \mathcal\pi_{\alpha(\gamma)}$
for all $\gamma\in\Gamma.$
\item  If $ \mathcal V_j$ is defined to be $T^{-j}(\mathcal V_0),$
then the collection $\{ \mathcal V_j\}$ is a GMRA relative to
$\mathcal\pi$ and $\mathcal\delta,$ where $\mathcal\delta= T^{-1}.
$
\item  The multiplicity function associated to the core subspace $ \mathcal V_0$ is
the given function $m,$ and the given $H$
is a filter constructed from the GMRA $\{\mathcal V_j\}.
$
\end{enumerate}
\end{theorem}
\begin{proof}
To prove the first claim, note that  by Proposition \ref{dilationoperator}, $\mathcal D^{-1}$ is an isometry from $\mathcal W_{n+1}$ onto $\mathcal W_{n}$.
  By Lemma \ref{Ruelle}, we also have that the definition of $T$ above gives an isometry from ${\mathcal V_0}\oplus \mathcal W_0$ onto $\mathcal V_0$.
The second claim follows immediately from the definitions, since the operators $S_H$, $S_G$ and $\mathcal D^{-1}$ all change the argument of the function from $\omega$ to $\alpha^*(\omega)$.
\par
Next, we show that the collection $\{\mathcal V_j\}$ is a GMRA.
The fact that $\mathcal V_j\subseteq \mathcal V_{j+1}$
follows from $T(\mathcal V_0)\subset \mathcal V_0$, which is immediate from the definition of $T$.
That $\mathcal V_{j+1} = \mathcal\delta(\mathcal V_j)$ follows immediately from the definition of $\mathcal\delta= T^{-1}.
$
  The trivial intersection property follows from our assumption that $S_H$ is a pure isometry, and the dense union from the fact noted in the previous paragraph that $T^{-1}\mathcal V_0=\mathcal V_0\oplus \mathcal W_0$ and $T^{-1} \mathcal W_n=\mathcal W_{n+1}$.
\par
  As a component of the direct sum space, $\mathcal V_0$ is clearly invariant under the multiplication operators $\omega(\gamma)$ that define the representation $\mathcal\pi$.
The given function $m$ is clearly the multiplicity function of that representation.
  To establish that $H$ is a corresponding filter, we note that we can take $J$ to be the identity for this $\mathcal V_0$, and calculate 
\begin{align*}
J\delta^{-1}J^{-1}(C_i)(\omega)&=T(C_i)(\omega)\cr
&=S_H(C_i)(\omega)\cr
&=\bigoplus_j h_{i,j}(\omega)\chi_{\sigma_i}(\alpha^*(\omega))\cr
&=\bigoplus _j h_{i,j}(\omega),
\end{align*}
where the last equality follows from the fact that by the filter equation, $h_{i,j}$ is supported on ${\alpha^*}^{-1}(\sigma_i)$.
\end{proof}
\begin{remark}
We will denote the GMRA $\{\mathcal V_j\}$ constructed above by
$\{V_j^{m,H,G}\}$
and refer to it as the {\bf canonical GMRA} having these parameters.
\end{remark}
In Section \ref{examples} we will construct canonical GMRAs related to classical examples, as well as new ones.
  First, we establish in the next section conditions under which two GMRAs are the same.  While our construction procedure requires the choice of a complementary filter $G$, we will see that the equivalence classes depend only on the two parameters $m$ and $H$.   
%%%%%%%Section 4
\section{A classifying set for GMRAs}
\label{equivalence}
Let $\{V_j\}$ be a GMRA in a Hilbert space $\mathcal H,$
relative to a representation $\pi$ of $\Gamma$ and a unitary operator $\delta,$
and let $\{V'_j\}$ be a GMRA in a Hilbert space $\mathcal H',$
relative to a representation $\pi'$ of $\Gamma$ and a unitary operator $\delta'.
$
\begin{definition}
\label{equiv}
We say that the GMRAs $\{V_j\}$ and $\{V'_j\}$ are
{\bf equivalent} if there exists a unitary operator
$U:\mathcal H \to \mathcal H'$ that satisfies:
\begin{enumerate}
\item $U(V_j) = V'_j$ for all $j.
$
\item $U\circ \pi_\gamma = \pi'_\gamma \circ U$ for all $\gamma\in\Gamma.
$
\item $U\circ \delta = \delta'\circ U.
$
\end{enumerate}
\end{definition}
For classical examples in $L^2(\mathbb R^d)$, the Fourier transform $\mathcal F$ gives an equivalence between any GMRA $\{V_j\}$ and $\{\widehat{V}_j\}$.
  Further, 
if an operator $U$ gives an equivalence between $\{V_j\}$ and $\{V'_j\}$, two GMRAs for dilation by $A$ and translation by $\mathbb Z^d$ in $L^2(\mathbb R^d)$, then $\widehat{U}=\mathcal F\circ U\mathcal F^{-1} $ is multiplication by a function $u$ with absolute value 1, and such that $u(A^{*j}\omega) = u(\omega)$ for all integers $j$ (\cite {Be}).    Thus equivalence between GMRAs for the same dilation in $L^2(\mathbb R^d)$ generalizes the notion of different MSF wavelets attached to the same wavelet set.
\par
Recall that we consider only GMRAs with a finite multiplicity function $m$ and with the associated measure $\mu$ absolutely continuous with respect to Haar measure.   
Our first aim is to prove that every such GMRA is equivalent to one of the canonical GMRAs constructed in the preceding section.
We will then describe the equivalence relation among these
GMRAs in terms of the parameters $m,$ $H$ and $G$.  We will need the following lemma.
   
\begin{lemma}\label{v0w0}
The GMRAs $\{V_j\}$ and $\{V'_j\}$ are equivalent
if and only if there exists a unitary operator
 $P$ mapping $V_0$ onto $ V'_0$ that satisfies:
\begin{enumerate}
\item $P\circ \pi_\gamma = \pi'_\gamma \circ P$ for all $\gamma\in\Gamma.
$
\item $P\circ \delta^{-1} = {\delta'}^{-1} \circ P.
$
\end{enumerate}
\end{lemma}
\begin{proof}
We first assume that the conditions above are satisfied, and show that $\{V_j\}$ and $\{V'_j\}$ are equivalent.
 For each $n\geq0,$ define an operator $Q_n:W_n\to W'_n$ by
\[
Q_n = {\delta'}^{n+1} \circ P \circ \delta^{-(n+1)}.
\]
Now, define $U=P\oplus\bigoplus_{n=0}^{\infty}Q_n$ on  $\mathcal H=V_0\oplus\bigoplus_{n=0}^{\infty} W_n.$
One checks directly that $U$ satisfies the required conditions.
\par
For the converse, assume that $\{V_j\}$ and $\{V'_j\}$ are equivalent, with $U:\mathcal H \to \mathcal H'$
implementing the equivalence.
 Define $P\;=\;U|_{V_0}.$
By the definition of equivalence, $P$ maps $V_0$ to $V'_0,$ and conditions (1) and (2) follow.  
\end{proof}
\begin{theorem}
\label{equiv1}
Let $\{V_j\}$ be a GMRA.
Let $m$ be its (finite) associated multiplicity function, and let $H=[h_{i,j}]$ be a filter constructed from the GMRA using the map $J$.  
Let $G$ be a complementary filter to $H.$
  Then the GMRA $\{V_j\}$ is equivalent to the canonical GMRA $\{V_j^{m,H,G}\}.
$
\end{theorem}
\begin{proof}
Define $P:V_0\mapsto\bigoplus L^2(\sigma_i)$ by $P=J.$ 
Condition (1) of Lemma \ref{v0w0} follows immediately.
The fact that $J\circ \delta^{-1} \circ J^{-1} = S_H$
proves the second condition of that lemma.
\end{proof}
\begin{theorem}
\label{equiv2}
The canonical GMRAs $\{V_j^{m,H,G}\}$ and $\{{V}_j^{m',H',G'}\}$
are equivalent if and only if $m=m',$ and there exists a matrix-valued function $A$ on $\widehat\Gamma$ such that
\begin{enumerate}
\item $A(\omega)= \left(\begin{matrix} A_1(\omega) & 0\\
0& 0\end{matrix}\right),$ where $A_1(\omega)$ is a unitary matrix of dimension $m(\omega).
$
\item $H(\omega)A^t(\omega) = A^t(\alpha^*(\omega))H'(\omega).
$
\end{enumerate}
\end{theorem}
\begin{proof}
Suppose first that $m=m'$ and that there exists a matrix-valued
function $A$ satisfying the conditions.  Let $\tau_r$ be the subset of $\widehat\Gamma$ on which $m(\omega) = m'(\omega)=r.$
Then both subspaces
$V_0^{m,H,G}$ and $V_0^{m',H',G'}$ are equal to
\[
\bigoplus_i L^2(\sigma_i) \equiv \bigoplus_r L^2(\tau_r,{\mathbb C}^r).
\]
Define $P:V^{m,H,G}_0\to
 V^{m',H',G'}_0$ by $[P(f)](\omega) = A(\omega)f(\omega)$.
It follows directly that $P$  satisfies the conditions of Lemma \ref{v0w0},
and hence $\{V_j^{m,H,G}\}$ and $\{{V}_j^{m',H',G'}\}$ are equivalent.
\par
Conversely, suppose an operator $P$ 
exists and satisfies
the conditions of Lemma \ref{v0w0}.
The first condition on $P$ implies that the two
representations of $\Gamma$ on $V^{m,H,G}_0$ and $V^{m',H',G'}_0$ are
unitarily equivalent, whence $m$ must equal $m',$
and $V^{m,H,G}_0=V^{m',H',G'}_0=\bigoplus_i L^2(\sigma_i) =\bigoplus_r L^2(\tau_r,{\mathbb C}^r).$
It is known (e.g. \cite{Be}) that any unitary operator $P$ on the direct sum of
vector-valued $L^2$ spaces that commutes
with all the multiplication operators $\gamma(\omega),$ is itself a multiplication operator of the form
\[
[P(f)](\omega) = A(\omega)f(\omega),
\]
where $A(\omega)= \left(\begin{matrix} A_1(\omega) & 0\\
0& 0\end{matrix}\right),$ and $A_1(\omega)$ is a unitary matrix whose dimension is $r=m(\omega)$ for $\omega\in\tau_r.$
The second condition of Lemma \ref{v0w0} then implies that  $A$ satisfies condition (2) of the theorem.
\end{proof}
\begin{cor}
Let $m$ be a multiplicity function and let $H$ be a filter
relative to $m$ and $\alpha^*$
for which $S_H$ is a pure isometry.
If $G$ and $G'$ are any two complementary filters
to $H,$ then the
GMRAs $\{V_j^{m,H,G}\}$ and $\{V_j^{m,H,G'}\}$ are equivalent.
\end{cor}
The preceding theorem introduces a notion of equivalence
among filters that we will use to
build a set of classifying parameters for the equivalence classes of GMRAs.  In the following definition, we use our knowledge of the form of $A$ to rewrite the equivalence using the conjugate transpose $A^*$.
\begin{definition}
\label{equivfilt}
Let $m$ be a multiplicity function.
Filters $H$ and $H'$ relative to $m$ and $\alpha^*$ are called
{\bf equivalent} if there exists a matrix-valued
function $A$ on $\widehat\Gamma,$ with $A(\omega)= \left(\begin{matrix} A_1(\omega) & 0\\
0& 0\end{matrix}\right),$ where $A_1(\omega)$ is a unitary matrix of dimension $m(\omega),
$ and such that
\[
H'(\omega) = A(\alpha^*(\omega)) H(\omega) A^*(\omega)
\]
for almost all $\omega\in\widehat\Gamma.$
\end{definition}
\begin{remark}
If $H$ and $H'$ are two filters constructed from the same GMRA using different Stone's Theorem operators $J$ and $J'$, then $H$ and $H'$ are equivalent according to this definition.  Here the matrix-valued function $A$ comes from the multiplication operator $J'J^{-1}$.
\end{remark} 
\begin{lemma}
Let $H$ be a filter relative to $m$ and $\alpha^*,$
and let $A$ be a matrix-valued function of the
form described in the preceding theorem.
Define the matrix-valued function $H'$ by
\[
H'(\omega) = A(\alpha^*(\omega)) H(\omega) A^*(\omega).
\]
Then $H'$ is a filter relative to $m$ and $\alpha^*,$
i.e., $H'$ satisfies the filter equation.
\end{lemma}
\begin{proof}
  We note that if we write $H_1(\omega)$ for the upper left $m(\alpha^*(\omega))\times m(\omega)$ block of $H(\omega),$ and let $\Lambda(\omega)$ be $N$ times the $m(\omega)\times m(\omega)$ identity, then the filter equation (\ref{filter1}) can be rewritten as
\begin{equation}
\label{filter1b}
\sum_{\alpha^*(\zeta)=\omega} H_1(\zeta){H}_1^*(\zeta)=\Lambda(\omega).
\end{equation}
We must show that if $H$ satisfies Equation (\ref{filter1b}), then so does $H'$.
We have
\begin{eqnarray*}
\sum_{\alpha^*(\zeta)=\omega} H'_1(\zeta){H'}_1^*(\zeta)&=&\sum_{\alpha^*(\zeta)=\omega}A_1(\omega) H_1(\zeta) A_1^*(\zeta)A_1(\zeta) H_1^*(\zeta) A_1^*(\omega)\\
&=&A_1(\omega)\Lambda(\omega)A_1^*(\omega)\\
&=&\Lambda(\omega)
\end{eqnarray*}
\end{proof}
Let $H$ be a filter relative to $m$ and $\alpha^*.$
In \cite{BFMP1} it was shown that
the operator $S_H$ fails to be
a pure isometry if and only if it has an eigenvector,
i.e., if and only if there exists an element $F\in \bigoplus L^2(\sigma_i)$
and a complex number $\lambda$
for which $H^t(\omega) F(\alpha^*(\omega)) = \lambda F(\omega),$
where $|\lambda|=1 = \|F(\omega)\|$
for almost all $\omega.$  Motivated by this result, we make the following definition:
\begin{definition}
A filter $H$ is called an {\bf eigenfilter} if there exists a constant $\lambda$ with $|\lambda|=1$ such that for almost all $\omega$, $H_{1,1}(\omega)=\lambda$ and $H_{1,j}(\omega)=0$ for $j>1$.
\end{definition}
\par
Using this definiton, we have the following restatement of the result from \cite{BFMP1}:
\begin{prop}
\label{pure}
$S_H$ fails to be a pure isometry if and only if
$H$ is equivalent to an eigenfilter.
\end{prop}
\begin{proof}
If there exists a matrix-valued function $A$ such that
\[
H'(\omega) A(\omega) = A(\alpha^*(\omega)) H(\omega),
\]
where $H'(\omega)$ is an eigenfilter,
then, computing the first rows of both sides, we see that
the first row of $A$ is the desired eigenvector $F.$
\par
Conversely, if $S_H$ has an eigenvector $F,$
build a unitary-valued matrix $A(\omega)$ having $F(\omega)$
as its first row.
Set $H'(\omega) = A(\alpha^*(\omega)) H(\omega) A^*(\omega).$
By the previous lemma, $H'$ is a filter relative to
$m$ and $\alpha^*.$
Moreover, one can see that $H'_{1,1} = \lambda.$
Because $H'$ is a filter, it follows that the elements $H'_{1,j}(\omega) $ are all 0 for $j>1.$
Hence, $H'$ has the desired form.
\end{proof}
\par
Now, let $S$ be the set of all pairs $(m,H),$ where
$m$ is a multiplicity function and $H$ is a filter relative to $m$ and $\alpha^*.$
Let $S_0$ be the subset of $S$ comprising those pairs
$(m,H)$ for which $H$ is  equivalent to an eigenfilter,
and let $S_1=S\setminus S_0.$
\par
Finally let $E = S_1/\equiv$ be the set of
equivalence classes of $S_1$
with respect to the equivalence relation
$(m_1,H_1) \equiv (m_2,H_2)$ if $m_1=m_2$ and $H_1$ is equivalent to $H_2.$
\begin{theorem}
The set $E$ is a classifying set for the
equivalence classes of GMRAs (with finite multiplicity functions and associated measures absolutely continuous with respect to Haar measure),
in the sense that there is a 1-1 correspondence between $E$ and the classes of GMRAs,
and this correspondence can be described explicitly.
\end{theorem}
\begin{proof}
Given an element $s\in E,$ let $(m,H)$ be a representative of the equivalence class $s.$
Let $G$ be a filter complementary to $H,$ and define $\kappa(s)$ to be the equivalence class of the GMRA $\{V^{m,H,G}\}$.
By Theorem \ref{equiv2}, the map $\kappa$ is both well defined and one-to-one, and by Theorem \ref{equiv1}, it is onto.
\end{proof}
   
%%%%%%%%%%%%%Section 5
\section{Examples}
\label{examples}
We will now use the technique outlined in Section \ref{construction} to construct examples of canonical GMRAs, and apply the ideas of the Section \ref{equivalence} to discuss their equivalence.
We work first in the classical setting of MRAs (so $m\equiv 1$) with single wavelets (so $\widetilde{m}\equiv 1$) for dilation by 2 in $L^2(\mathbb R)$.
\begin{example}
\label{MRAdil2}
Any MRA for dilation by 2 in $L^2(\mathbb R)$ with $m=\widetilde{m}\equiv 1$ has 
canonical Hilbert space 
\begin{equation}
\label{dil2}
L^2(\mathbb T)\oplus L^2(\mathbb T)\oplus \left(\bigoplus_{j=1}^{\infty} L^2(2^j\mathbb T)\right)=V^{m,H,G}_0\oplus W^{m,H,G}_0\oplus\left (\bigoplus_{j=1}^{\infty}W^{m,H,G}_j\right)
\end{equation}
with $\pi_n(\bigoplus f_l)=e_n(\bigoplus f_l),$ where $e_n(x)=e^{2\pi i n x}$, and
\begin{eqnarray}
\label{candil2}
\lefteqn{\delta^{-1}(f_{V_0}\oplus f_{W_0}\oplus \left(\bigoplus_{j=1}^{\infty} f_{W_j}\right)}\\
\nonumber&=&\left(h(\omega)f_{V_0}(2\omega)+g(\omega)f_{W_0}(2\omega)\right)\oplus\sqrt{2}f_{W_1}(2\omega)\oplus\left(\bigoplus_{j=2}^{\infty}\sqrt 2 f_{W_j}(2\omega)\right).
\end{eqnarray}
Equivalence for two different MRAs with single wavelets is equivalent to the existence of a period 1 function $a$ such that $|a(\omega)|=1$ and $h'(\omega)=a(2\omega)h(\omega)\overline{a(\omega)},$ where $h$ and $h'$ are filters constructed from the two MRAs. Thus, in particular, equivalence requires that $|h|=|h'|.$  However, this is not sufficient, as we will see below.  Determining which filters give equivalent MRAs requires determining exactly which functions on the 1-torus are coboundaries where cohomological equivalence is given by Definition \ref{equivfilt}.
\par
For the Shannon MRA, with $\widehat{V}_0=L^2([-\frac12,\frac12])$, we have 
$h=\sqrt2\chi_{[-\frac14,\frac 14]}$ and 
 $g=\sqrt{2}\chi_{\pm[\frac14,\frac12]}$ in the above formula.
  By mapping $W_j^{m,H,G}\mapsto L^2(\pm 2^j[\frac12,1])$, we can map this canonical GMRA to the Fourier transform of the Shannon GMRA.
 \par
 For the Haar MRA, with  $V_0$ spanned by translates of $\chi_{[0,1]}$, we have $h=\frac1{\sqrt2}(1+e_{-1})$, $g=\frac1{\sqrt2}(e_{-1}-1)$ in the above formula.
  Here there is no obvious mapping between the canonical GMRA and either the original or its Fourier transform.
  However, we know all three are equivalent by Theorem \ref{equiv1}.
 \par
 In either the Shannon  or Haar examples, we can switch the roles of $h$ and $g$ to get a new MRA that cannot be realized in $L^2(\mathbb R)$ (since iterating the refinement equation (\ref{refine}) leads to a scaling function that must be identically 0).
  The canonical Hilbert space will still be given by Equation (\ref{dil2}), and the operators $\pi_{n}$ will be as above.
  However, in the dilation formula (\ref{candil2}), we will now have $h$ given by the old $g$, and $g$ by the old $h$.
   Proposition \ref{pure} shows that we still have $S_h$ a pure isometry, so that the canonical construction does produce a GMRA.
  Looking at the $V^{m,h,g}_{-j}$ and $W^{m,h,g}_{-j}$ that result in the case of the reversed Shannon GMRA shows how this example differs from Shannon MRA itself: 
$$\delta^{-1}:V_0^{m,h,g}=L^2\left(\left[-\frac12,\frac12\right]\right)\mapsto L^2\left(\pm\left[\frac14,\frac12\right]\right)\mapsto L^2\left(\pm\left[\frac38,\frac12\right]\right)\mapsto\cdots$$
$$\delta^{-1}:W_0^{m,h,g}=L^2\left(\left[-\frac12,\frac12\right]\right)\mapsto L^2\left(\left[-\frac14,\frac14\right]\right)\mapsto L^2\left(\pm\left[\frac14,\frac38\right]\right)\mapsto\cdots$$
\par
Since the absolute values of the filters in the three examples discussed here are all different on sets of positive measure, the three are seen to be inequivalent MRAs.  To see that for MRAs with wavelets, the filters having equal absolute value almost everywhere is not sufficient for equivalence, consider the MRA built from $h'=-h$, where $h$ is the filter for the Haar example.  A simple Fourier analysis argument shows that there is no solution to  $h'(\omega)=a(2\omega)h(\omega)\overline{a(\omega)},$ so this MRA must be inequivalent to the Haar MRA.  We note that it has the same canonical Hilbert space as Haar, and the same subspaces $V_j,$ but its dilation on $V_0$ is the negative of the Haar dilation.  This negative sign causes problems in the iteration of the refinement equation, so this example cannot be realized in $L^2(\mathbb R)$.  
\par A third example in this setting begins with the Cohen filters $h = \frac{1}{\sqrt{2}}(1 + e_{-3})$ and $g = \frac{1}{\sqrt{2}}(1 - e_{-3})$.  The infinite product construction which follows from the refinement equation in $L^2(\mathbb R)$ yields the functions $\phi = \frac{1}{3}\chi_{[0,3)}$ and $\psi = \frac{1}{3} (\chi_{[0, \frac{3}{2})} - \chi_{[\frac{3}{2}, 3)})$, which fail to be an orthonormal scaling function and orthonormal wavelet, respectively, since neither has orthonormal translates.  However, it can be shown that the negative dilate space (for dilation by 2) of the Cohen Parseval wavelet coincides with that of the Haar orthonormal wavelet.  Hence, the Cohen GMRA equals the Haar MRA.
\par We may apply Theorem \ref{gmraconstr} to the Cohen filters and multiplicity functions $m\equiv 1$, $\widetilde m\equiv 1$.  The canonical Hilbert space will be that given by Equation (\ref{dil2}), on which the the integers act by multiplication by exponentials.  We see that the spaces $V_j$ for the canonical Cohen GMRA are the same as those for the canonical Haar MRA whenever $j\geq 0$.  However, since $\frac{1}{\sqrt{2}}(1 + e_{-3})$ and $\frac{1}{\sqrt{2}}(1 + e_{-1})$ have different moduli, the two filters must be inequivalent.  Therefore the two canonical GMRAs must be inequivalent.  
\par Lastly, we remark that while the Cohen wavelet $\psi = \frac{1}{3} (\chi_{[0, \frac{3}{2})} - \chi_{[\frac{3}{2}, 3)})$ is only a Parseval wavelet in $L^2(\R)$, the element $(0, \chi_{[-\frac12,\frac12)}, 0, 0, 0, \ldots)$ is an orthonormal wavelet for the canonical Hilbert space (\ref{dil2}), with respect to $\pi_n$ and $\delta$ defined by Equation (\ref{candil2}) using the Cohen filters.    Dutkay et. al. (\cite{BDP},\cite{DJ2}) also produced an orthonormal wavelet from the Cohen filter, using a "super-wavelet" construction.  The associated GMRA in $L^2(\mathbb R)\oplus L^2(\mathbb R)\oplus L^2(\mathbb R)$  can be seen to be equivalent to our canonical Cohen GMRA by defining the map $P$ in Lemma \ref{v0w0} in the natural way to take the $n^{th}$ translate of the Dutkay scaling function to $e^{2\pi i n x}$ in the canonical $V_0$.     Of course, this example cannot be realized in $L^2(\R)$.
 \end{example} 
Next, we consider two non-MRA examples for dilation by 2 in $L^2(\mathbb R)$:  the Journ\'{e} GMRA, and the example for the Journ\'{e} multiplicity function with low-pass filter of rank $a=2$ described in \cite{IJKLN}, Example 13.
  As is noted there, a GMRA cannot be constructed for this second example using the infinite product construction.
  However, by Proposition \ref{pure}, the construction of this paper can be carried out to give such a GMRA.
\begin{example}
Let $m$ be  the multiplicity function corresponding to the Journ\'e wavelet:
 \[
m(x)=\begin{cases}2&\text{if $x\in [-\frac{1}{7},\frac{1}{7})$}\\
1&\text{if  $x\in \pm [\frac{1}{7},\frac{2}{7})\cup \pm[\frac{3}{7},\frac{1}{2})$}\\
0&\text{otherwise,}
\end{cases}
\]
so $\sigma_1=[-\frac12,-\frac37]\cup[-\frac27,\frac27]\cup[\frac37,\frac12]$ and $\sigma_2=[-\frac17,\frac17]$.
 Since we know the Journ\'{e} GMRA has an associated single orthonormal wavelet, $\widetilde{m}\equiv 1$.
 
Filters that give rise to the Journ\'{e} wavelet via the infinite product construction are described in \cite{Cou}, \cite{BCM}, and \cite{BJMP}.
  In particular, we may take 
$$H=\left(\begin{array}{ll}\sqrt{2}\chi_{[-\frac27,-\frac14]\cup[-\frac17,\frac17)\cup[\frac14,\frac27]}&0\\\sqrt2\chi_{[-\frac12,-\frac37]\cup[\frac37,\frac12]}&0\end{array}\right)\mbox { and }$$
$$G=\left(\begin{array}{ll}\sqrt2\chi_{[-\frac14,-\frac17]\cup[\frac17,\frac14]}&\sqrt2\chi_{[-\frac17,\frac17]}\end{array}\right)$$
Here $V_0^{m,H,G}=L^2(\sigma_1)\oplus L^2(\sigma_2)$, and $W_j^{m,H,G}=L^2(2^j\mathbb T)$, $j\geq0$.
This canonical GMRA can be mapped to the usual Journ\'{e} GMRA by integrally translating $\sigma_1$ and $\sigma_2$ to the scaling set to form $V_0$, and $\mathbb T\equiv[-\frac12,\frac12]$ to the wavelet set to form $W_0$.
\par   In \cite{IJKLN}, an alternative filter $H'$ for the same multiplicity function, but which satisfies the low-pass condition of rank $a=2$ is constructed:
\[
h'_{1,1}=\sqrt{2}\chi_{[-\frac{2}{7},-\frac{1}{4})\cup [-\frac{1}{7},\frac{1}{7})\cup [\frac{1}{4}, \frac{2}{7})},\quad  h'_{1,2}=h'_{2,1}= 0, \quad \text{and}\quad h'_{2,2}=\sqrt{2}\chi_{[-\frac{1}{14},\frac{1}{14})}.
\] 
By partitioning $\mathbb R/\mathbb Z$ as described in \cite{Cou} and \cite{BCM}, we can build the following complementary filter $G'$:
\[
g'_{1,1}=\sqrt{2}
\chi_{[-\frac12,-\frac37)\cup[-\frac14,-\frac17)\cup[\frac17,\frac14)\cup[\frac37,\frac12)},\quad
g'_{1,2}=\sqrt{2}\chi_{[-\frac17,-\frac1{14})\cup[\frac1{14},\frac17)}.
\]
The spaces $V_0^{m,H',G'}$ and $W_j^{m,H',G'}$ for $j\geq 0$ are the same as those for the canonical GMRA corresponding to the standard Journ\'{e} filters.  However, the different filters in the rank 2 example will change the dilation, and thus change the spaces $V_{j}^{m,H',G'}$ and $W_{j}^{m,H',G'}$ for $j<0$.  For example, we have 
$V_{-1}^{m,H',G'}=L^2\left(\left[-\frac17,\frac 17\right]\cup\pm\left[\frac14,\frac27\right]\right)\oplus L^2\left(\left[-\frac1{14},\frac1{14}\right]\right),$ while the standard $V_{-1}^{m,H,G}=L^2\left(\left[-\frac17,\frac 17\right]\cup\pm\left[\frac14,\frac27\right]\cup\pm\left[-\frac3{7},\frac1{2}\right]\right)\oplus 0.$
The fact that all the $V_{-j}^{m,H',G'}$ allow nonzero second components with support overlapping that of the first component suggests the impossibility of mapping the rank 2 example into $L^2(\mathbb R$) as we mapped the standard example.  Indeed, since iterating the refinement equation would lead to a scaling function with a degenerate multiplicity function (\cite{IJKLN}), the rank 2 example cannot be realized in $L^2(\mathbb R)$.  Thus, these two examples must not be equivalent.  
\end{example}
For our next example, we consider dilation by 3, both in $L^2(\mathbb R)$ and in the Dutkay/
\\Jorgensen enlarged Cantor fractal space (\cite{DJ}).
\begin{example}
The MRA Haar 2-wavelet for dilation by 3 in $L^2(\mathbb R)$  has canonical Hilbert space  
$$L^2(\mathbb T)\oplus \left(L^2(\mathbb T)\oplus L^2(\mathbb T)\right)\oplus\left(\bigoplus_{j=1}^{\infty} L^2(3^j\mathbb T)\oplus L^2(3^j\mathbb T)\right).
$$
The canonical $\delta^{-1}=S_h\oplus \left(S_{g_1}\oplus S_{g_2}\right)\oplus\left(\bigoplus_{j=1}^{\infty}\mathcal D^{-j}\right)$, where
 $$h=\frac1{\sqrt 3}(1+e_1+e_2),\quad g_1=\frac1{\sqrt 2}(e_1-e_2)\mbox { and }g_2=\frac1{\sqrt 6}(-2+e_1+e_2),$$
 and $\mathcal D^{-j}(f_1\oplus f_2)(\omega)= \sqrt 3^j (f_1\oplus f_2)(3^j\omega).
$
\par
 The Cantor set MRA has the same canonical GMRA except with
 $$h=\frac 1{\sqrt2}(1+e_2),\quad g_1=e_1 \mbox{ and }g_2=\frac1{\sqrt2}(1-e_2).
$$
These two examples must be inequivalent since their h"s have different absolute values.  The latter cannot be realized in $L^2(\mathbb R)$, since $h(0)=\sqrt 2$, so that the iterated refinement equation (\ref{refine}) would again force the scaling function to be identically 0.
     
\end{example}
Our final example uses a group $\Gamma$ different from $\mathbb Z^d$.
\begin{example}
Let
$\Gamma_j = \bigoplus_{i=j}^{\infty}
[\Z_2]_i=\bigoplus_{i=j}^{\infty}\{1,-1\}_i,$ embedded as a subgroup of
$D=\Gamma_{-\infty}=\bigoplus_{i=-\infty}^{\infty} [\Z_2]_i$ by
$\Gamma_j=\bigoplus_{i=-\infty}^{j-1}\{1\}_i\oplus
\bigoplus_{i=j}^{\infty} [\Z_2]_i.$  Let $\alpha$ be defined on
$\Gamma_0$ by $\alpha(\gamma)_n = \gamma_{n-1}$ for $n>0$ and
$\alpha(\gamma)_0 = 1.$  Let $\mathcal H = l^2(D),$ and let
$\pi$ be the restriction to $\Gamma_0$ of the regular representation of
$D.$   Define $S$ on $D$ by $[S(d)]_n = d_{n-1},$
and note that $S(\gamma) \equiv  \alpha(\gamma)$ for
$\gamma\in\Gamma_0.$
Define $\delta$ on $\mathcal H = l^2(D)$ by
$[\delta(f)](d) = f(S(d)),$
and note that
 $\delta^{-1}\pi_\gamma \delta = \pi_{\alpha(\gamma)}.$
 \par  We have 
$\Gamma_{j+1}\subseteq \Gamma_j,$ and
$\cap_{j=-\infty}^{\infty}\Gamma_j=\{e_D\}$, 
so that $l^2(\Gamma_{j+1})\subseteq l^2(\Gamma_j)$ and
$\cap_{j=-\infty}^{\infty}l^2(\Gamma_j)=l^2(\{e_D\}),$ where
$e_D=(\cdots,\;1,\;1,\;1,\;\cdots)$ denotes the additive identity
element of $D=\Gamma_{-\infty}.$ 
 If we let $V_j=l^2(\Gamma_{-j}),$ then $\{V_j\}$ is almost a GMRA.
It fails only because constant multiples of the function
$\chi_{\{e_D\}}$ belong to $\cap V_j.$  We will make it into a GMRA by
tensoring it with the dilation by 2 Haar GMRA.  It is known (as in
\cite{JM})
that the tensor product of two GMRAs gives a GMRA.  By tensoring our
almost GMRA with an actual one, we will preserve all the properties of
the almost GMRA, and eliminate the non-trivial intersection.   
\par Accordingly, let $\Gamma'=\mathbb Z$ act in $\mathcal
H'=L^2(\mathbb R)$ by $\pi'_nf(x)=f(x-n)$, and let $\delta'f(x)=\sqrt2
f(2x)$.  We have $\alpha'$ acting on $\Gamma'$ by $\alpha'(n)=2n$.
  Write $\{V'_j\}$ for the usual Haar GMRA that results from taking
$V'_0$ to be the closed linear span of translates of $\chi_{[0,1]}$.
  Set $\mathcal H''=\mathcal H\otimes\mathcal H'$, equipped with the
representation $\pi\times \pi'$ of $\Gamma''=\Gamma_0\times\Gamma'$ and
the operator $\delta\otimes\delta'$.  We let
$\alpha''=\alpha\times\alpha'$ and note that $\alpha''^*$ acts on
$\widehat{\Gamma''}=\prod_{i=0}^{\infty} [\Z_2]_i\times \mathbb T$ by
$\alpha''^*((\omega_0,\omega_1,\omega_2,\cdots)\times
x)=(\omega_1,\omega_2,\cdots)\times 2x$, where we parameterize $\mathbb
T$ by $[-\frac 12,\frac 12).$  We have $N=4,$ and
$\ker(\alpha''^*)=(\{-1,1\}\times \prod_{i=1}^{\infty}
[\{1\}]_i)\times\{0,\frac12\}$.  To build the dilation described in
Section \ref{construction}, we can take the cross section 
 $c((\omega_0,\omega_1,\omega_2,\cdots)\times
x)=(1,\omega_0,\omega_1,\cdots)\times \frac x2.$ We have $m\equiv1$ and
$\widetilde{m}\equiv3$, so
$\sigma_1=\widetilde\sigma_1=\widetilde\sigma_2=\widetilde\sigma_3=\prod
_{i=0}^{\infty} [\Z_2]_i\times \mathbb T.$  

 We define our filter $H=h_1\otimes h_2,$ where $h_1$ is the filter on
$\prod_{i=0}^{\infty} [\Z_2]_i$ given by
$h_1=\sqrt{2}\chi_{\{1\}_0\times \prod_{i=1}^{\infty} [\Z_2]_i},$ and
$h_2$ is the low-pass filter for the Haar GMRA described in Example
\ref{MRAdil2}, that is $h_2=\frac1{\sqrt2}(1+e_{-1})$.  For our filter
complementary to $H,$ we define  $g_1=\sqrt{2}\chi_{\{-1\}_0\times
\prod_{i=1}^{\infty} [\Z_2]_i}$ and let
 $g_2$ be the high-pass filter for the Haar GMRA,
$g_2=\frac1{\sqrt2}(e_{-1}-1)$.  We then take our complementary filter
$G$ to be the matrix whose rows are $h_1\otimes g_2$, $g_1\otimes h_2$,
and $g_1\otimes g_2$.  

For an alternative GMRA, 
we can replace $h_1$ by $h'_1=\chi_{\{1\}_0\times \prod_{i=1}^{\infty}
[\Z_2]_i} -
\chi_{\{-1\}_0\times \prod_{i=1}^{\infty} [\Z_2]_i}$
and $g_1$ by $g'_1=-\chi_{\{1\}_0\times \prod_{i=1}^{\infty} [\Z_2]_i}+
\chi_{\{-1\}_0\times \prod_{i=1}^{\infty} [\Z_2]_i}.$
These could be viewed as more fractal-like when combined with the $h_2$
and $g_2$ in the standard
tensor product construction.
\end{example}

%%%%%%%%bibiliography

\end{document}